\documentclass{article}

\usepackage{amsmath,amsthm}
\usepackage{geometry}
\usepackage{amssymb}
\usepackage{lscape}
\usepackage[all,cmtip]{xy}
\usepackage{cancel}
\usepackage{bm}
\usepackage{cite}
\usepackage{multirow}
\usepackage{framed}
\usepackage{hyperref}
\usepackage{xcolor}
\usepackage{tikz, bbm}
\usepackage{tikz-cd}
\usepackage{xparse}
\usepackage{ifthen}
\usepackage{etoolbox}

\newcommand{\R}{\mathbb{R}}
\newcommand{\Z}{\mathbb{Z}}

\newtheorem{theorem}{Theorem}
\newtheorem*{mrule}{Rule}
\newtheorem{lemma}[theorem]{Lemma}
\theoremstyle{definition}
\newtheorem{property}[theorem]{Property}

\newcommand{\conv}{\operatorname{conv}}
\newcommand{\inv}{^{-1}}

\newcommand{\CSep}{-1.5}

\newcommand{\CoxeterLabelHeight}{0.15}
\newcommand{\CoxeterCrossSize}{0.08}
\newcommand{\CoxeterNodeSize}{2.5}
\newcommand{\StyleLabel}[1]{\tiny $#1$} 

\NewDocumentCommand\drawit
{ m m > { \SplitList { , } } m > { \SplitList { , } } m}
   {
     \begin{tikzpicture}[every node/.style={inner sep=\CoxeterNodeSize, fill=white}]
       \foreach \i in {2,...,#1} {
            \node at (\i-0.5,\CoxeterLabelHeight) {\StyleLabel{\GetElement{#2}{\i-1}}};
            \draw (\i-1,0) -- (\i,0);
       }
       \foreach \i in {1,...,#1} { \Squarecross{\i} }
       \ProcessList{#3}{ \Square }
       \ProcessList{#4}{ \Circleold }
     \end{tikzpicture}
   }

\newcommand{\Circleold}[1] {   \ifthenelse{\equal{#1}{0}}{}{\node[draw, inner sep=\CoxeterNodeSize, circle,fill=white] at (#1, 0) {}; }}
\newcommand{\Circle}[1] {   \node[draw, inner sep=\CoxeterNodeSize, circle,fill=white] at (#1, 0) {}; }
\newcommand{\CirclE}[2] {   \node[draw, inner sep=\CoxeterNodeSize, circle,fill=white] at (#1, #2) {}; }
\newcommand{\Squarecross}[1] {
    \node[draw, fill=white, inner sep=\CoxeterNodeSize, rectangle] at (#1, 0) {};
    \draw (#1-\CoxeterCrossSize, -\CoxeterCrossSize) -- (#1+\CoxeterCrossSize, \CoxeterCrossSize); \draw (#1-\CoxeterCrossSize, \CoxeterCrossSize) -- (#1+\CoxeterCrossSize, -\CoxeterCrossSize);
}
\newcommand{\Squarecrossed}[2] {
    \node[draw, fill=white, inner sep=\CoxeterNodeSize, rectangle] at (#1, #2) {};
    \draw (#1-\CoxeterCrossSize, #2-\CoxeterCrossSize) -- (#1+\CoxeterCrossSize, #2+\CoxeterCrossSize); \draw (#1-\CoxeterCrossSize, #2+\CoxeterCrossSize) -- (#1+\CoxeterCrossSize, #2-\CoxeterCrossSize);
}
\newcommand{\Square}[1] { \ifthenelse{\equal{#1}{0}}{}{\node[draw, fill=white, inner sep=\CoxeterNodeSize, rectangle] at (#1, 0) {};} }
\newcommand{\SquarE}[2] { \node[draw, fill=white, inner sep=\CoxeterNodeSize, rectangle] at (#1, #2) {}; }

\ExplSyntaxOn
\NewDocumentCommand{\GetElement}{m m}
{
    \clist_gclear_new:c {g_Coxeter_clist}
    \clist_gset:cn {g_Coxeter_clist} {#1}
    \clist_item:cn {g_Coxeter_clist} {#2}
}
\ExplSyntaxOff

\title{Classifying Regular Polyhedra and Polytopes using Wythoff's Construction}
\author{Spencer Whitehead}

\begin{document}
\maketitle
\begin{abstract}
Wythoff's construction associates a uniform polytope to a Coxeter diagram whose vertices are decorated with crosses, which indicate the subgroup stabilizing a generic point.
Champagne, Kjiri, Patera, and Sharp remarked in \cite{Champagne1995} that by associating more information to a Coxeter diagram, one can furthermore determine the types and number of faces of such a polytope.
This article provides a proof of the result in \cite{Champagne1995}, and uses it to provide a classification of the regular polytopes.
\end{abstract}

\section{Introduction}

A Wythoffian polytope is a uniform polytope generated as the convex hull of the orbit of a point in a finite group generated by reflections.
Wythoffian polytopes are interesting because they comprise a large number of uniform polytopes, and also are easy to work with due to their particularly simple method of construction.
It was remarked in \cite{Champagne1995} that one may use the Coxeter diagram associated to a Wythoffian polytope to answer questions about a Wythoffian polytope such as the number and kind of faces it possesses.
This method allows one to quickly determine what a Wythoffian polytope ``looks like'' in space, given its depiction as a decorated Coxeter diagram.
As we see in Section 6, it also enables one to easily enumerate the subset of Wythoffian polytopes meeting some additional combinatorial property; for example, having all its $2$-faces be equal.

The purpose of this paper is to provide a proof of the remarks in \cite{Champagne1995}, to use them to fully describe the abstract structure of a Wythoffian polytope as a face lattice whose elements are pairs of symmetries and decorated Coxeter diagrams, and then finally to provide a classification of the regular polytopes in this setting.
As a consequence, we also see a complete list of ways the regular polytopes arise from Wythoff constructions (for example, the 24-cell can be constructed out of kaleidoscopes of shape $B_4, D_4$, and $F_4$).

Similar enumerations of regular polytopes (\cite[\S 7]{Coxeter-polytopes}, for example) and how they arise as Wythoff constructions are well known, and Coxeter's ``simplical subdivision'' is a method of determining the symmetry group of a regular polytope not dissimilar to that of Theorem~\ref{theorem:regularwythoff}.
Rather than inductively building the entire symmetry group, the approach in this paper is to show that the symmetry conditions on a regular polytope force the symmetry group to contain enough reflections to generate it, which in turn means that regular polytopes are Wythoffian.
The method in the present paper is a demonstration of how these data associated to a uniform polytope may be used to recover structural information in a regular way, and could concievably be used to tackle similar problems; for example, listing the Wythoffian polytopes with certain kinds of faces or transitivity properties.

Section 2 introduces basic definitions relating to polytopes and regularity, as well as the abstract properties of polytopes as lattices that will be required.
Section 3 describes Wythoff's construction.
Section 4 describes how regular polytopes may be realized as Wythoffian polytopes.
Section 5 develops the basic theory of decorated Coxeter diagrams used in \cite{Champagne1995} with proof, and uses the resulting combinatorial information to describe the face lattice structure of a Wythoffian polytope.
Section 6 uses the results of the previous two sections to classify the regular polytopes.

\subsection*{Acknowledgements}

The author would like to thank Benoit Charbonneau for many productive discussions about this paper, and to two anonymous referees whose suggestions were greatly helpful in improving the quality of the work.

\section{Definitions}

We follow \cite{Ziegler1995} in providing basic definitions surrounding polytopes.
A \emph{(convex) polytope} is the convex hull of a finite set of points in $\R^n$, and it is said to be \emph{degenerate} if it is contained in a hyperplane; when not otherwise specified, we assume that polytopes are non-degenerate.
When $c \in \R^n, c_0 \in \R$, an inequality $c^T x \le c_0$ is said to be \emph{valid} for a polytope $P$ if each point in $P$ satisfies it.
A \emph{face} of $P$ is a set of the form $P \cap \{x \in \R^n : c^T x = c_0\}$, where $c^T x \le c_0$ is valid for $P$.
The \emph{dimension} of a face is the dimension of its affine hull.
A \emph{$k$-face} is a $k$-dimensional face; typical names are \emph{vertex} for a $0$-face, \emph{edge} for a $1$-face, \emph{ridge} for a $(n-2)$-face, and \emph{facet} for a $(n-1)$-face.

A \emph{(complete) flag} in $P$ is a chain of inclusions $F_0 \subseteq F_1 \subseteq \cdots \subseteq F_n$, where $F_i$ is an $i$-dimensional face of $P$.
In a polytope, any partial flag can be completed to a flag of this form.
Two flags are said to be \emph{adjacent} if they differ in exactly one position.

A \emph{regular polytope} is a polytope whose group of symmetries acts transitively on its flags.
In three dimensions, the regular polytopes are the familiar platonic solids: the tetrahedron, cube, octahedron, dodecahedron, and icosahedron.

To a polytope $P$ we associate a polar dual $P^* := \{ c \in \R^n : c^T x \le 1 \text{ for all } x \in P\}$.
Recall that if $0 \in P$ then $P = (P^*)^*$.
When $P$ is a regular polytope, one can show that the vertices of $P^*$ are the centroids of the facets of $P$ (or an appropriate scaling of them).

All polytopes satisfy the following two properties (\cite[Chap. 1]{Abstract-Regular-Polytopes}):
\begin{property}[Diamond property]
\label{property:diamond}
Given a $(k-1)$-face $F_{k-1}$ and a $(k+1)$-face $F_{k+1}$ of a polytope, there exist exactly two distinct $k$-faces contained in $F_{k+1}$ and containing $F_{k-1}$.
\end{property}

\begin{property}[Flag-connectedness]
\label{property:strongflag}
Given two flags $F, F'$ of a polytope, there exists a sequence of flags $F = F_1, F_2, \ldots, F_{n-1}, F_n = F'$ having $F_i$ adjacent to $F_{i+1}$ for all $i$.
\end{property}

The diamond property says, for example, that each edge contains two vertices, and that each vertex is incident to exactly two edges of every face it is contained in.
A consequence of the diamond property is that each flag is adjacent to exactly $n$ other flags: if  $n-1$ elements of in the flag are fixed, the only choice for the unfixed element is the other element in the relevant diamond.
See Fig.~\ref{fig:propertyex} for an example of both these properties.

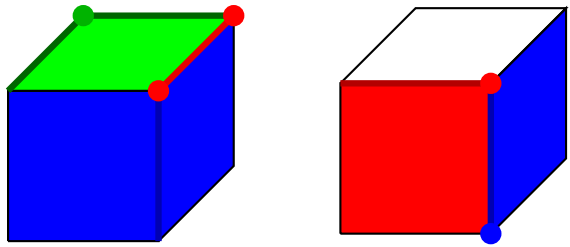
\begin{figure}
\centering
\scalebox{2}{
\begin{tikzpicture}
\draw[fill=blue] (0,0) -- (1,0) -- (1,1) -- (0,1) -- (0,0);
\draw[fill=blue] (1,0) -- (1,1) -- (1.5,1.5) -- (1.5,0.5) -- (1,0);
\draw[fill=green] (0,1) -- (1,1) -- (1.5,1.5) -- (0.5,1.5) -- (0,1);
\draw[color=black!30!blue,very thick] (1,0) -- (1,1);
\draw[color=black!60!green,very thick] (0.5,1.5) -- (0,1);
\draw[color=black!60!green,very thick] (0.5,1.5) -- (1.5,1.5);
\node [draw, inner sep=1pt, circle, fill=black!30!green, color=black!30!green,very thick] at (0.5,1.5) {};
\draw[color=black!10!red,very thick] (1,1) -- (1.5,1.5);
\node [draw, inner sep=1pt, circle, fill=red, color=red,very thick] at (1.5,1.5) {};
\node [draw, inner sep=1pt, circle, fill=red, color=red,very thick] at (1,1) {};
\end{tikzpicture}
\hskip0.5cm
\begin{tikzpicture}
\draw[fill=red] (0,0) -- (1,0) -- (1,1) -- (0,1) -- (0,0);
\draw[fill=blue] (1,0) -- (1,1) -- (1.5,1.5) -- (1.5,0.5) -- (1,0);
\draw (0,1) -- (0.5, 1.5) -- (1.5,1.5) -- (1,1);
\draw[color=black!30!red,very thick] (0,1) -- (1,1);
\draw[color=black!30!blue,very thick] (1,0) -- (1,1);
\node [draw, inner sep=1pt, circle, fill=red, color=red,very thick] at (1,1) {};
\node [draw, inner sep=1pt, circle, fill=blue, color=blue,very thick] at (1,0) {};
\end{tikzpicture}
}
\caption{The left image demonstrates the diamond property. With $k=0$ (red), each edge has two vertices. With $k=1$ (green), each vertex is incident to two edges in a face. With $k=2$ (blue), each edge borders two faces.
The right shows two flags in a cube: to get from one to the other, first change the blue vertex to the red vertex, then the blue face to the red face, then the blue edge to the red edge.}
\label{fig:propertyex}
\end{figure}

\section{Wythoff's Construction}
\label{sec:wythoff}

A \emph{uniform polytope} is a relaxation of the symmetry conditions on a regular polytope from flag-transitivity to vertex-transitivity.
In two dimensions, the uniform polygons are defined to be exactly the regular polygons.
In higher dimensions, a polytope is said to be uniform if all its facets are uniform, and if its group of symmetries acts transitively on its vertices.
Given a finite reflection-generated subgroup of $O(n)$ and a point in $\R^n$, we consider the convex polytope whose vertices are the orbit of the point under the group.
Coxeter showed in \cite{Coxeter-WythoffConstruction} that a judicious choice of point results in a uniform polytope---the uniform polytopes arising from this construction are called \emph{Wythoffian}.
Almost all known uniform polytopes arise from the kaleidoscopic construction of Wythoff and Coxeter: Conway and Guy showed in \cite{Conway-Guy-Four-Dimensional-Archimedean} that of the 63 four-dimensional uniform polytopes not occurring in infinite families, only three are not Wythoffian.

\begin{figure}
\centering
\scalebox{0.7}{
\begin{tikzpicture}[every node/.style={inner sep=\CoxeterNodeSize, fill=white}]
\node at (-1,0) {\small $A_n$};
\draw (0,0) -- (2.5,0);
\draw (3.5,0) -- (4,0);
\SquarE{0}{0}
\SquarE{1}{0}
\SquarE{2}{0}
\SquarE{4}{0}
\node at (3,0) {\scalebox{0.5}{$\cdots$}};

\node at (-1, \CSep) {\small $B_n$};
\node at (0.5, \CSep+\CoxeterLabelHeight) {\StyleLabel{4}};
\draw (0,\CSep) -- (2.5,\CSep);
\draw (3.5,\CSep) -- (4,\CSep);
\SquarE{0}{\CSep}
\SquarE{1}{\CSep}
\SquarE{2}{\CSep}
\SquarE{4}{\CSep}
\node at (3,\CSep) {\scalebox{0.5}{$\cdots$}};

\node at (-1, 3*\CSep) {\small $I_2(k)$};
\node at (0.5, 3*\CSep+\CoxeterLabelHeight) {\StyleLabel{k}};
\draw (0, 3*\CSep) -- (1, 3*\CSep);
\SquarE{0}{3*\CSep}
\SquarE{1}{3*\CSep}

\node at (-1, 2*\CSep) {\small $D_n$};
\draw (0,2*\CSep) -- (2.5,2*\CSep);
\draw (3, 2*\CSep) -- (3, 2.5*\CSep);
\draw (2.5,2*\CSep) -- (4,2*\CSep);
\SquarE{0}{2*\CSep}
\SquarE{1}{2*\CSep}
\SquarE{3}{2*\CSep}
\SquarE{4}{2*\CSep}
\SquarE{3}{2.5*\CSep}
\node at (2,2*\CSep) {\scalebox{0.5}{$\cdots$}};

\node at (6, 0) {\small $H_3$};
\node at (7.5, \CoxeterLabelHeight) {\StyleLabel{5}};
\draw (7,0) -- (9, 0);
\SquarE{7}{0}
\SquarE{8}{0}
\SquarE{9}{0}

\node at (6, \CSep) {\small $H_4$};
\node at (7.5, \CSep+\CoxeterLabelHeight) {\StyleLabel{5}};
\draw (7,\CSep) -- (10, \CSep);
\SquarE{7}{\CSep}
\SquarE{8}{\CSep}
\SquarE{9}{\CSep}
\SquarE{10}{\CSep}

\node at (6, 2*\CSep) {\small $F_4$};
\node at (8.5, 2*\CSep+\CoxeterLabelHeight) {\StyleLabel{4}};
\draw (7,2*\CSep) -- (10, 2*\CSep);
\SquarE{7}{2*\CSep}
\SquarE{8}{2*\CSep}
\SquarE{9}{2*\CSep}
\SquarE{10}{2*\CSep}

\node at (12, 0) {\small $E_6$};
\draw (13, 0) -- (17, 0);
\draw (15, 0) -- (15, 0.5*\CSep);
\SquarE{13}{0}
\SquarE{14}{0}
\SquarE{15}{0}
\SquarE{16}{0}
\SquarE{17}{0}
\SquarE{15}{0.5*\CSep}

\node at (12, \CSep) {\small $E_7$};
\draw (13, \CSep) -- (18, \CSep);
\draw (16, \CSep) -- (16, 1.5*\CSep);
\SquarE{13}{\CSep}
\SquarE{14}{\CSep}
\SquarE{15}{\CSep}
\SquarE{16}{\CSep}
\SquarE{17}{\CSep}
\SquarE{18}{\CSep}
\SquarE{16}{1.5*\CSep}

\node at (12, 2*\CSep) {\small $E_8$};
\draw (13, 2*\CSep) -- (19, 2*\CSep);
\draw (17, 2*\CSep) -- (17, 2.5*\CSep);
\SquarE{13}{2*\CSep}
\SquarE{14}{2*\CSep}
\SquarE{15}{2*\CSep}
\SquarE{16}{2*\CSep}
\SquarE{17}{2*\CSep}
\SquarE{18}{2*\CSep}
\SquarE{19}{2*\CSep}
\SquarE{17}{2.5*\CSep}
\end{tikzpicture}
}
\caption{A list of the irreducible Coxeter diagrams}
\label{fig:coxeterdiagram}
\end{figure}
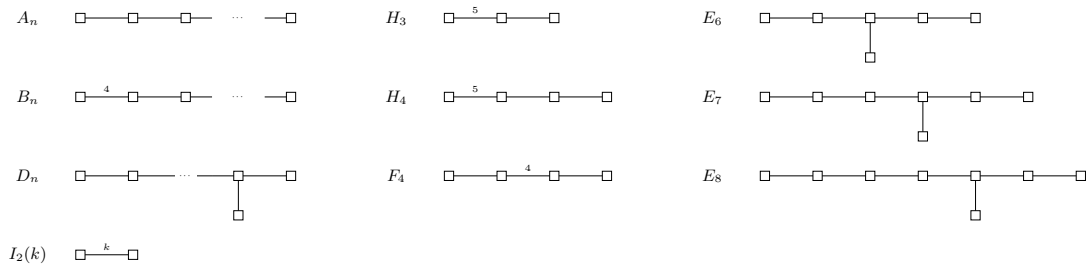
An abstract finite group generated by involutions is called a \emph{Coxeter group}.
Coxeter showed in \cite{Coxeter-DiscreteGroupsGeneratedByReflections} that these groups have presentations of the form
\[ \left\langle r_1, \ldots, r_n | (r_i r_j)^{m_{ij}} = 1, i, j = 1, \ldots n \right\rangle \]
for some symmetric positive integer matrix $(m_{ij})$ with $m_{ii} = 1$ for all $i$.
Such a group can be modelled in $O(n)$ by associating to each $r_i$ a reflection through a hyperplane in $\R^n$ containing the origin.
To ensure the condition $(r_i r_j)^{m_{ij}} = 1$, the hyperplanes associated to $r_i$ and $r_j$ should meet at dihedral angle $\frac{\pi}{m_{ij}}$.
This information is encoded into a \emph{Coxeter diagram} consisting of a vertex for every mirror, and an edge between mirrors $i, j$ with label $m_{ij}$.
Conventionally, if $m_{ij} \le 2$ no edge is added, and if $m_{ij} = 3$ the label is omitted.
Coxeter's classification showed that the Coxeter diagram of any Coxeter group is a disjoint union of the diagrams in Fig.~\ref{fig:coxeterdiagram}

To pick a point for Wythoff's construction, select some subset $S$ of these mirrors to be a stabilizer, and choose the unique point of unit length that lies on every mirror of $S$, and is of equal non-zero distance from every mirror not in $S$.
The resulting polytope is uniform, and moreover, all points resulting in uniform polytopes for this group arise from this procedure.
Two examples of decorated Coxeter diagrams and the polytopes they represent are given in Fig.~\ref{fig:coxeterexample}; the convention that mirrors in $S$ are drawn with a cross through them is used.
If $S$ contains a full connected component of the Coxeter diagram, any mirror in that component fixes the whole polytope, and thus the polytope is contained within a hyperplane; conversely, a Wythoffian polytope is contained in a hyperplane only if its Coxeter diagram has a connected component contained in $S$.

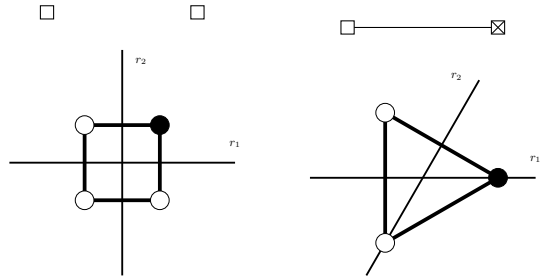
\begin{figure}
\centering
\scalebox{0.5}{
\begin{tikzpicture}
\scalebox{2}{
\SquarE{-1}{2}
\SquarE{1}{2}
}
\draw[line width=0.5mm] (3,0) -- (-3, 0);
\draw[line width=0.5mm] (0,3) -- (0, -3);
\draw[line width=1mm] (1,1) -- (-1,1) -- (-1,-1) -- (1,-1) -- (1,1);
\node[circle,draw,fill=black,inner sep=5pt] at (1, 1){};
\node[circle,draw,fill=white,inner sep=5pt] at (-1,1){};
\node[circle,draw,fill=white,inner sep=5pt] at (-1,-1){};
\node[circle,draw,fill=white,inner sep=5pt] at (1,-1){};
\node at (3,0.5) {\small $r_1$};
\node at (0.5,2.7) {\small $r_2$};
\end{tikzpicture}
}
\hspace{0.5cm}
\scalebox{0.5}{
\begin{tikzpicture}
\scalebox{2}{
\draw (-1,2) -- (1,2);
\SquarE{-1}{2}
\Squarecrossed{1}{2}
}
\draw[line width=0.5mm] (-3,0) -- (3,0);
\draw[line width=0.5mm] (-1.5,-3*0.866) -- (3*0.5,3*0.866);
\draw[line width=1mm](2,0) -- (-1, -2*0.866);
\draw[line width=1mm] (2,0) -- (-1, 2*0.866);
\draw[line width=1mm] (-1, 2*0.866) -- (-1, -2*0.866);
\node[circle,draw,fill=black,inner sep=5pt] at (2, 0){};
\node[circle,draw,fill=white,inner sep=5pt] at (-1,-2*0.866){};
\node[circle,draw,fill=white,inner sep=5pt] at (-1,2*0.866){};
\node at (3,0.5) {\small $r_1$};
\node at (0.9,2.7) {\small $r_2$};
\end{tikzpicture}
}
\caption{Two examples of decorated Coxeter diagrams and the polytopes they represent---the left is a square, and the right is an equilateral triangle. The thick lines are the reflecting mirrors, and the white points are reflections of the black point in the mirrors.}
\label{fig:coxeterexample}
\end{figure}

\section{Regular polytopes are Wythoffian}
\label{sec:regularwythoff}

\begin{theorem}
\label{theorem:regularwythoff}
Let $P \subseteq \R^n$ be a regular polytope.
Then $P$ is Wythoffian.
\end{theorem}

\begin{proof}
To the polytope $P$ we may associate its \emph{centroid}, obtained by averaging its vertices.
A symmetry of the polytope is an affine isometry that permutes the vertices, and hence fixes the centroid.
Assume without loss of generality that $P$ is centered at the origin, so symmetries of $P$ are orthogonal transformations.

As previously mentioned, assume $P$ is non-degenerate; if it were degenerate, it is non-degenerate in its affine hull, where the below proof applies.
If some facet contains the origin, then $P$ is contained in a half-space of the form $\{ x \in \R^n : a^T x \ge 0\}$.
Since $P$ is not contained in the hyperplane $a^T x = 0$, for each vertex $v$ we have $a^T v \ge 0$, and for at least one vertex, $a^T v > 0$.
Then if $c$ is the center of $P$, we find that $a^T c > 0$.
But $c$ is the origin, so $a^T c = 0$.
This is a contradiction; thus no facet contains the origin.
Since each lower dimensional face is contained in a facet, it follows that no face of any dimension contains the origin.

Let $R$ be a ridge of $P$.
Then $R$ is contained in an affine subspace of dimension $n-2$, and does not contain the origin, so it has linear span of dimension $n-1$.
Denote by $H$ the unique hyperplane containing the origin spanned by $R$.
By Property~\ref{property:diamond}, $R$ is contained in exactly two facets, say $F$ and $F'$.
By flag transitivity, there exists a symmetry $s_R$ of $P$ fixing $R$ and sending $F$ to $F'$.
Since $s_R$ fixes $R$, it fixes $H$ by linearity.
Because $s_R$ is an isometry and is not the identity, it is uniquely determined to be reflection in $H$.

Let $W$ be the subgroup of the symmetries of $P$ generated by all the $s_R$.
Since $P$ has finite symmetry group, $W$ is finite also.
By Property~\ref{property:strongflag}, given any two facets $F, F'$ there is a sequence of adjacent flags $\Phi_1, \ldots, \Phi_n$ such that $\Phi_1$ has $F$ as a facet, and $\Phi_n$ has $F'$ as a facet.
Write $F_i$ for the facet of $\Phi_i$, and $R_i$ for the ridge of $\Phi_i$.
When $F_i \neq F_{i+1}$, adjacency of $\Phi_i$ and $\Phi_{i+1}$ guarantees that $R_i = R_{i+1}$, so that $F_{i+1} = s_{R_i} F_i$.
If $F_i = F_{i+1}$, let $w_i = 1$.
Otherwise, let $w_i = s_{R_i}$.
Then $w_n \cdots w_1 F = F'$, so $W$ acts transitively on the facets of $P$.

Pick a facet $F$ of $P$.
Since $F$ is also a polytope of one lower dimension, we may pick the center $x$ of $F$, and consider the set $Q = \conv(Wx)$.
As $W$ is finite, $Q$ is a polytope.
In fact, $Q$ is the convex hull of the centers of the facets of $P$, and is therefore dual to $P$, since $P$ is regular.
Since the dual of a regular polytope is regular, $Q$ is regular, and in particular, uniform.
Since $Q$ is a uniform polytope generated by taking the orbit of a point under a finite reflection group, it is Wythoffian.

What we have shown is that when $P$ is regular, the dual of $P$ is Wythoffian.
Consequently, if $P$ is a regular polytope containing the origin, then $P^*$ is regular and so $(P^*)^* = P$ is Wythoffian.
\end{proof}

\section{Decorated Coxeter Diagrams}

The purpose of this section is to provide a rigorous statement and proof of the theory of decorated Coxeter diagrams first developed in \cite{Champagne1995}.
These results may also be extended to explicitly describe the entire face lattice structure of a Wythoffian polytope, which is the content of Theorem~\ref{thm:bigthm}.

Let $G$ be a Coxeter group with Coxeter diagram $D$, and let $S$ be any set.
Denote the vertices of $D$ by $V(D)$.
An \emph{$S$-decorated Coxeter diagram} (or just a \emph{decoration} or \emph{decorated Coxeter diagram} when $S$ is understood) is a map $f : V(D) \to S$. We will often think of the vertices of $D$ as being elements of $G$, or, when a representation of $G$ is fixed, the planes that they are reflections through or their normal vectors.

Using the language of decorated Coxeter diagrams, Wythoff's construction can be restated.

\begin{theorem}
Let $G \subseteq O(n)$ be a discrete group generated by reflections with normals $v_1, \ldots, v_n$ and with Coxeter diagram $D$, and let $f : V(D) \to \{0,1\}$ be a decorated Coxeter diagram.
There is a unit vector $x \in \R^n$ that lies on the planes $v_i^T y = 0$ where $f(i) = 0$ and has equal non-zero distance from the planes $v_i^T y = 0$ where $f(i) = 1$.
Then the convex hull $P=\conv(Gx)$ of the $G$-orbit of $x$ is a uniform polytope.
The symmetries of $P$ in $G$ are the group generated by $v_i$ where $f(v_i) = 0$.
\end{theorem}

By associating more detailed information to a Coxeter diagram, it is possible to obtain stronger information about the structure of a Wythoffian polytope.
Let $f : V(D) \to \{0,1\}$ be a decorated Coxeter diagram, giving a Wythoffian polytope $P$.
We can also think of $f$ as a $\{0,1,2\}$-decorated Coxeter diagram.
For such diagrams, we define the following transformation rule $(\star)$, introduced in \cite{Champagne1995}:

\begin{mrule}[Rule $(\star)$]
Given a $\{0,1,2\}$-decorated Coxeter diagram $f$ and an element $w \in V(D)$ such that $f(w) = 1$, define a new $\{0,1,2\}$-decorated Coxeter diagram $f'$ in the following manner: Set $f'(v) = 2$ exactly when $f(v) = 2$ or $v = w$, and set $f'(v) = 1$ exactly when $f(v) = 1$ and $v \neq w$, or when $f(v) = 0$ and $v, w$ are adjacent in $D$.
\end{mrule}

One can think of a $\{0,1,2\}$-decorated Coxeter diagram as a graph where vertices with a zero decoration are drawn as crossed-out boxes, vertices with a one decoration are drawn as boxes, and vertices with a two decoration are drawn as circles.
The rule $(\star)$ then says that a $\{0,1,2\}$-decorated Coxeter diagram may be transformed by selecting a box, replacing it with a circle, and removing the crosses from any adjacent box.

The significance of decorations obtained by iterated applications of this rule is that they fully classify the faces of a Wythoffian polytope: after $k$ applications of the rule to the original decoration $f$ giving a decoration $f'$, let $S = (f')\inv(2)$, and then $f|_S : V(D|_S) \to \{0,1\}$ gives a polytope by Wythoff's construction applied to the subgroup $G|_S$ of $G$ generated by the $v_i$ contained in $S$. This polytope, which we denote $P(f,f')$, will be a $k$-dimensional face of the original polytope $P$. Moreover, every $k$-face of $P$ arises in this way as some sequence of $k$ applications of the rule $(\star)$.

For the rest of the section, fix $G \subseteq O(n)$ to be a Coxeter group with diagram $D$, and let $f$ be a decoration of $D$ yielding a uniform polytope $P$ via Wythoff's construction.
Let $x$ be the point selected in the kaleidoscope for Wythoff's construction.
Furthermore, $k$ will be some integer between $1$ and $n$, $f'$ will be a decoration of $D$ obtained by applying $k-1$ times the rule $(\star)$ to $f$, and $f''$ the decoration of $D$ obtained by applying $(\star)$ one further time to $f'$ on the vertex $v \in G$ of $D$.
For any decoration obtained from $f$ by applying the transformation rule some number of times, we think of $P(f,g)$ as being a subset of $P$, obtained by taking the convex hull of the orbit of $x$ under the group $G|_S$, where $S = g\inv(2)$.

\begin{lemma}
\label{lem:rightdim}
The polytope $P(f,f'')$ is contained in a $k$-dimensional affine hyperplane.
\end{lemma}
\begin{proof}
After $k$ applications of $(\star)$, the decoration $f''$ must have $|(f'')\inv(2)| = k$, so the resulting polytope is $k$-dimensional by Wythoff's construction.
\end{proof}

\begin{lemma}
The polytope $P(f,f'')$ is a $k$-face of $P$.
\end{lemma}
\begin{proof}
It is easiest to do this proof starting at $k = n$ and proceeding backwards; by considering what a decoration must look like before $(\star)$ is applied for the $n$th time, we see that it corresponds exactly to the decorations $f'$ where $f'(v) = 1$ for some vertex $v$ and $f'(w) = 2$ for each other $w$.
The goal is to prove that $P(f,f')$ is a top-dimensional face of $P$.
To this end, note that the centroid of $P(f,f')$ lies on the line through the origin and the vertex $W$ of the Weyl chamber corresponding to the uncircled element of the Coxeter diagram.
Therefore it cannot be the case that $P(f,f')$ separates a vertex from the rest of the polytope; such a vertex would have to be in a Weyl chamber incident to $W$, and the vertices of those chambers are already known to be in $P(f,f')$.

Since $P(f,f'')$ is Wythoffian, its facets are uniform, and we may continue to apply this result inductively to realize any $P(f,f''')$ as a face of $P$.
\end{proof}

Next, we will investigate which elements of $G$ send $P(f,f'')$ to itself: they are exactly the elements that can be expressed as a product of generators $v_i$ satisfying $f''(v_i) \neq 1$.
First, if $f''(v_i) = 2$, then it is clear that $v_i$ is a symmetry of $P(f,f'')$.
Also, if $f''(v_i) = 0$, then $v_i x = x$, and as well $(\star)$ guarantees that $v_i$ commutes with the subgroup of $G$ that generates $P(f,f'')$.
In particular, any diagram obtained from applying $(\star)$ always satisfies that generators with decoration 2 and generators with decoration 0 are not adjacent, and hence commute.
It follows that $v_i$ fixes $P(f,f'')$.

It remains to show that any $v_i$ with $f''(v_i) = 1$ is not a symmetry of $P(f,f'')$.
This is a consequence of Lemma~\ref{lem:rightdim}; applying the transformation rule to $v_i$ must increase the affine dimension of $P(f,f'')$ and so it cannot be the case that $v_i P(f,f'') = P(f,f'')$.
In conclusion, we have seen that
\begin{lemma}
The elements $v \in G$ that have $vP(f,f'') = P(f,f'')$ are exactly the elements of the subgroup generated by $v_i \in V(D)$ where $f''(v_i) \neq 1$.
\end{lemma}

Denote this group by $G(f'')$.
Using this, we may for example compute the number of faces of $P$ isomorphic to $P(f,f'')$ using the orbit-stabilizer theorem (and keeping in mind that two different decorations could potentially yield isomorphic faces).

\begin{lemma}
Every $k$-face of $P$ containing $x$ arises as $P(f,f'')$ for some $f''$.
\end{lemma}
\begin{proof}
Suppose not, and select $k$ to be as small as possible such that there is a face $F$ containing $x$ which is not equal to $P(f,f'')$ for any decoration $f''$.
One may show that the edges containing $x$ all take the form $\{x,vx\}$ where $v$ is a generating reflection, and so $k \ge 2$; also, we may take $k < n$, because the decoration $f'' \equiv 2$ identically recovers $P$.
Let $F'$ be a $(k-1)$-face of $F$ containing $x$, and choose a decoration $f'$ so that $F = P(f,f')$.
Since $k < n$, there is an edge of $F$ that contains $x$ but is not contained in $F'$.
Write this edge as $\{x, vx\}$ where $v \in G$ is a generating reflection.
Note that since $vx \neq x$, $f'(v) \neq 0$, and since $vx \not\in F'$, $f'(v) \neq 2$, so $f'(v) = 1$.

Then if $G'$ is the subgroup of $G$ generating $F'$ (that is, generated by $v \in V(D)$ with $f'(v) = 2$) we saw that the decoration $f''$ obtained by applying $(\star)$ to $f'$ and $v$ gives a face $F'' = P(f,f'')$ that contains $F'$ and $vx$ and has affine dimension $k$.
But $F$ also satisfies these properties, and $P$ is convex and non-degenerate, so an affine $k$-plane cannot contain two distinct $k$-faces.
Thus $F = F''$ arises as $P(f,f'')$.
\end{proof}

These results allow one to completely describe the combinatorial structure of the vertex figure of a uniform polytope, because they completely describe the incidence information of faces containing the vertex $x$.
The next step is to use vertex transitivity and information about the stabilizer to stitch together these vertex figures into a single lattice for the entire polytope.

Supposing that $P(f,f'')$ is some $k$-face of the polytope containing $x$, for an element $g \in G$, then $gP(f,f'')$ will be a $k$-face containing $gx$.
Moreover, since $P$ is vertex transitive, this allows the identification of the vertex figure structure at $gx$ with  that at $x$ acted upon by $g$.
Now it is clear that the same decoration based at $x$ and $gx$ correspond to an identical face of the polytope if and only if $g$ is a symmetry of the face.
This also preserves the incidence relation: it follows from the previous lemmas that $P(f,f'')$ has $P(f,f')$ as a $(k-1)$-face if and only if $f''$ may be obtained by applying $(\star)$ to $f'$ and some generator.

In short, we have seen the following description of the combinatorial structure of a Wythoffian polytope, which may be viewed as a reformulation of the result of \cite{Champagne1995}:

\begin{theorem}
\label{thm:bigthm}
Let $P$ be a Wythoffian polytope corresponding to a decoration $f : V(D) \to \{0,1\}$, where $D$ is a Coxeter diagram corresponding to a Coxeter group $G$.
The face lattice of $P$ admits the following description: the rank $k$ elements are pairs $(g, f')$ where $f'$ is a decorated Coxeter diagram obtained by applying the rule $(\star)$ to $f$ a total of $k$ times, and $g$ is an element of the coset space $G/G(f')$.

The incidence structure is defined by the relation that $(g, f')$ contains $(g,f'')$ if and only if $f''$ is obtained from $f'$ by one application of the rule $(\star)$.
\end{theorem}

One can also think of applying the rule $(\star)$ ``backwards'', by deleting vertices from a decorated Coxeter diagram. The theorem then says that $(n-k)$-faces are obtained by deleting $k$ vertices from a $\{0,1\}$-decorated Coxeter diagram, in such a way that there is no connected component on which the decoration is identically zero.

Fig~\ref{fig:ckpsexample} gives an example of the rule $(\star)$ being applied to determine faces of a polyhedron.
Here the convention $0 = $ cross, $1 = $ box, $2 = $ circle is used.

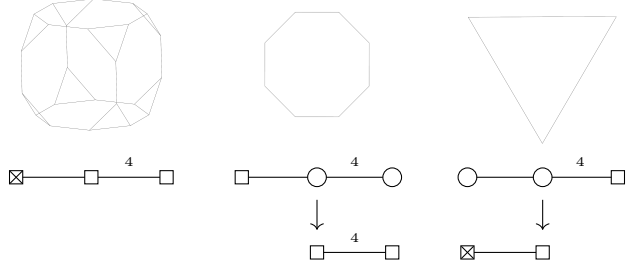
\begin{figure}
\centering

\begin{tikzpicture}
\node at (0,0.5) {\includegraphics[width=3cm]{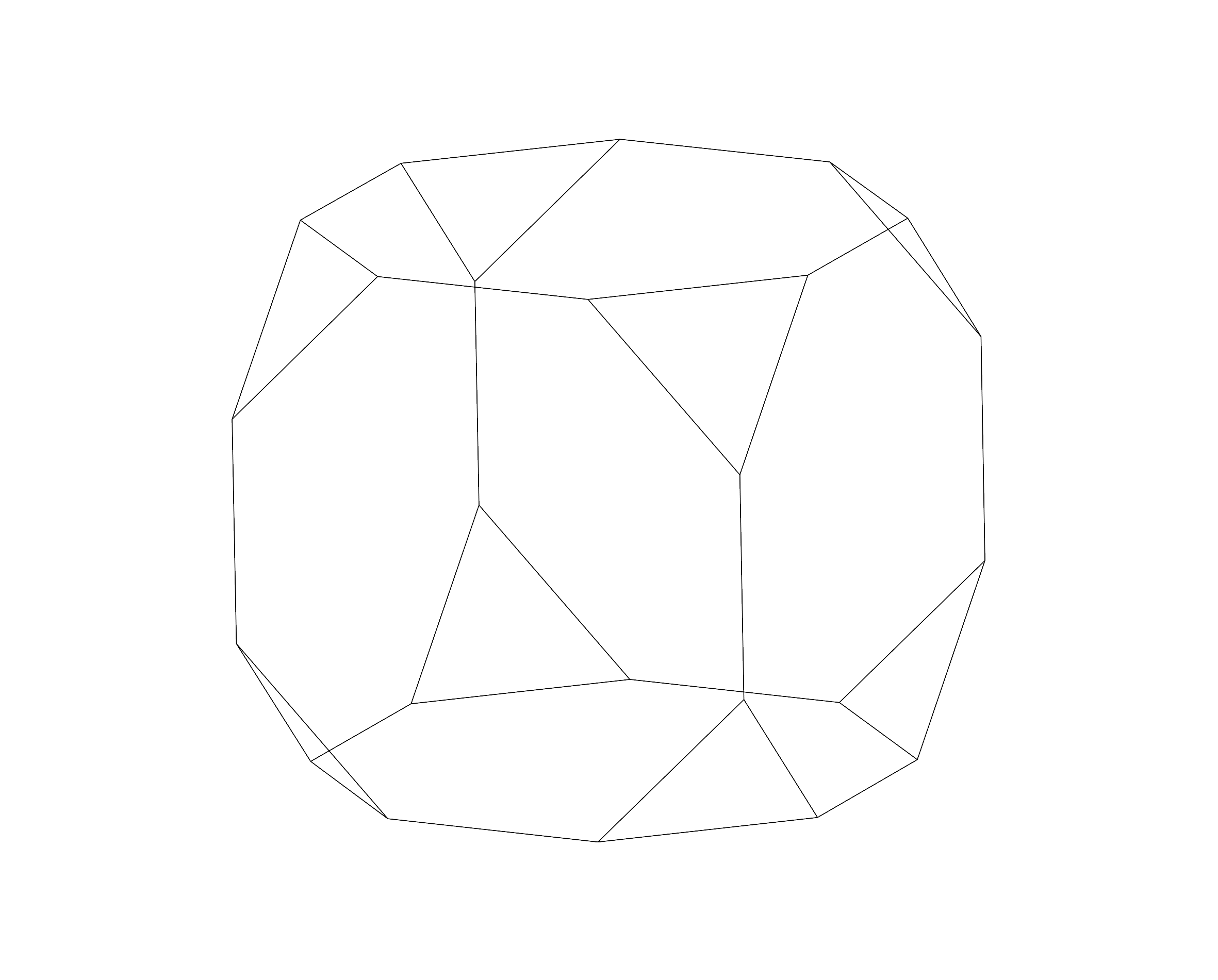}};
\draw (-1,-1) -- (1,-1);
\Squarecrossed{-1}{-1}
\SquarE{0}{-1}
\SquarE{1}{-1}
\node at (0.5, -0.8) {\tiny 4};

\node at (3,0.5) {\includegraphics[width=3cm]{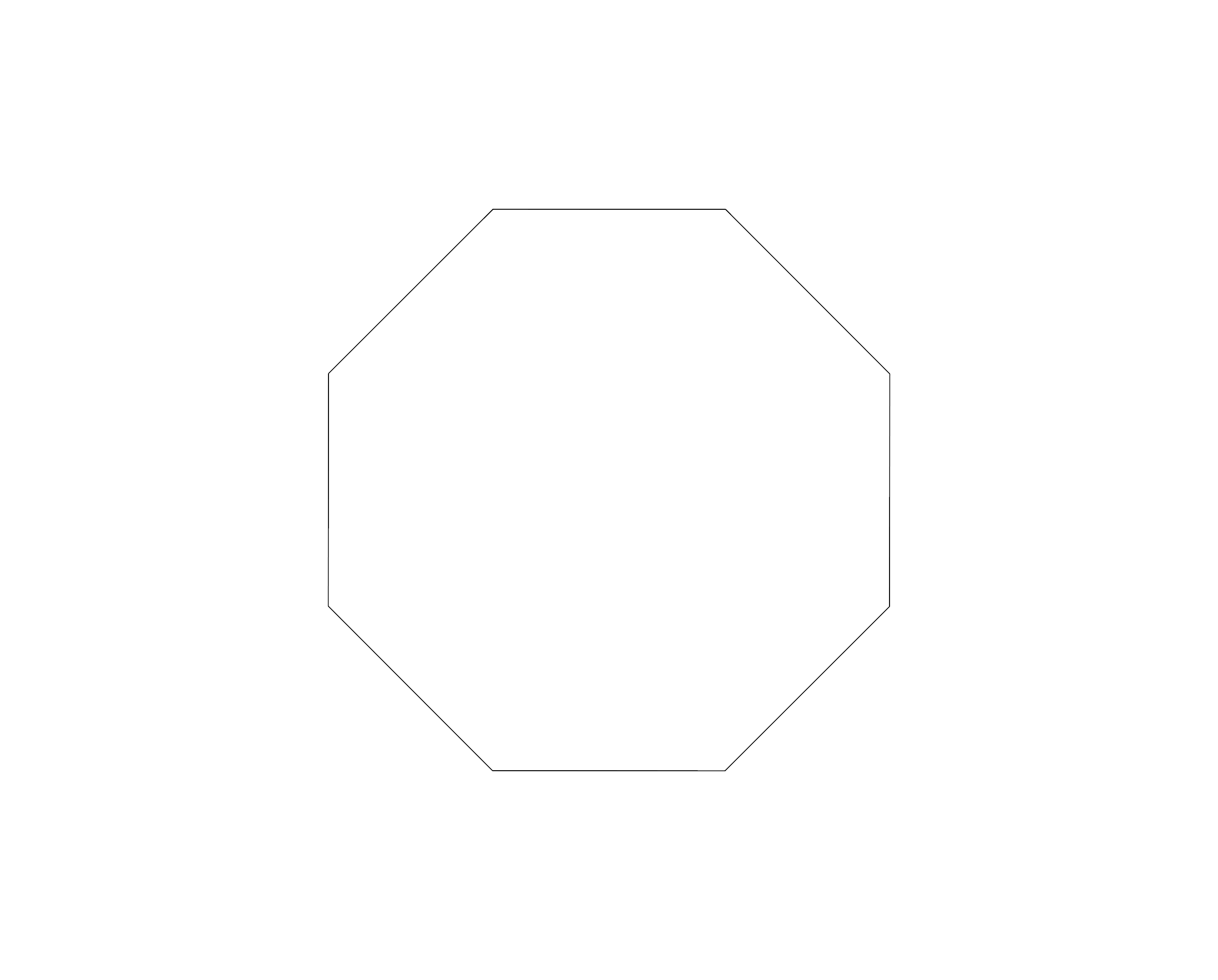}};
\draw (2,-1) -- (4,-1);
\draw (3,-2) -- (4,-2);
\draw [->](3,-1.3) -- (3,-1.7);
\SquarE{2}{-1}
\CirclE{3}{-1}
\CirclE{4}{-1}
\SquarE{3}{-2}
\SquarE{4}{-2}
\node at (3.5, -0.8) {\tiny 4};
\node at (3.5, -1.8) {\tiny 4};

\node at (6,0.5) {\includegraphics[width=3cm]{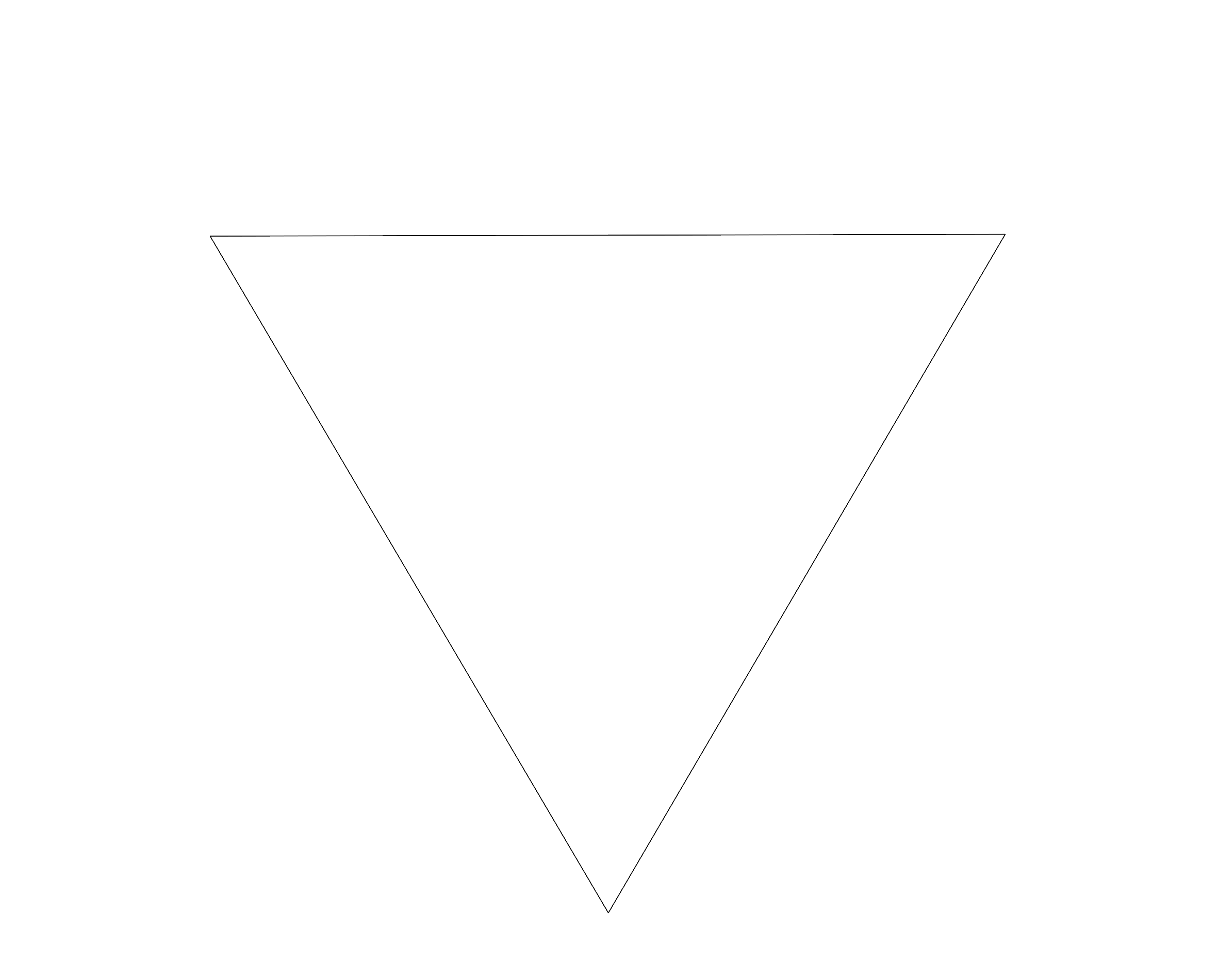}};
\draw (5,-1) -- (7,-1);
\draw (5, -2) -- (6,-2);
\draw [->](6,-1.3) -- (6,-1.7);
\CirclE{5}{-1}
\CirclE{6}{-1}
\SquarE{7}{-1}
\Squarecrossed{5}{-2}
\SquarE{6}{-2}
\node at (6.5, -0.8) {\tiny 4};

\end{tikzpicture}
\caption{Two decorations of \drawit{3}{$ $,4}{2,3}{0} giving $2$-faces are shown, corresponding to an octagon and a triangle. The triangle decoration is obtained by first applying $(\star)$ to the middle vertex. The leftmost vertex is then uncrossed by the transformation rule, so it may have $(\star)$ applied to it next.}
\label{fig:ckpsexample}
\end{figure}

\section{Classifying the regular polytopes}

Fig~\ref{fig:ckpsexample} provides some insight as to how the classification of regular polytopes can be executed in the language of decorated Coxeter diagrams: if by any series of applications of $(\star)$ to an initial decoration it is possible to obtain two non-isomorphic faces, then the corresponding polytope cannot be regular.

As an example, if $D$ is a connected Coxeter diagram and $f$ is a $\{0,1\}$-decoration of it yielding a \emph{regular} polytope $P$, we may restrict the size of the set $f\inv(1)$.
In particular, if the dimension of $P$ is at least 3 (so $D$ has at least 3 vertices), then $|f\inv(1)| = 1$.

\begin{lemma}
\label{lem:oneuncrossed}
If $f$ is a $\{0,1\}$-decoration of a connected Coxeter diagram $D$, and the associated polytope $P$ is uniform of dimension at least 3, then there is exactly one $v \in V(D)$ so that $f(v) = 1$.
\end{lemma}
\begin{proof}
Certainly there is at least one $v \in V(D)$ so that $f(v) = 1$, or else the case is degenerate.

Suppose that $u, v \in V(D)$ are distinct and have $f(u) = f(v) = 1$, and pick a third vertex $w$ adjacent to one of them.
If $u, w, v$ is a path in $D$, then $u$ and $v$ are not adjacent, since the Coxeter diagrams are acylic.
Applying $(\star)$ to $u$ and then $v$ gives a square face; so applying $(\star)$ to $u$ and then $w$ as well as $v$ and then $w$ must both also give square faces.
Since $u$ and $w$ are connected, this must mean that in the initial decoration $f(w) = 0$ and $uw$ has order 4, and likewise that $vw$ has order 4.
But no connected Coxeter diagram has two labels equal to 4, and so this cannot happen.

If $u, w,v$ is not a path in $D$, then $w$ is not adjacent to $v$.
Therefore, if $f(w) = 1$, then first applying $(\star)$ to $v$ and $w$ gives a square face, while applying it to $u$ and $w$ instead gives a $2k$-gon face where $k$ is the order of $uw$.
As they do not commute, $k > 2$ and so these faces cannot be isomorphic.

If on the other hand $f(w) = 0$, we may apply $(\star)$ to $u$ and then $w$, since $u$ and $w$ are adjacent, so after the first application the decoration will set $f'(w) = 1$.
This corresponds to a $k$-gon face, where $k$ is the order of $uw$.
By applying $(\star)$ instead to $u$ and $v$, we obtain a $2\ell$-gon face, where $\ell$ is the order of $uv$.
If $P$ is to be regular, it must be the case then that $2\ell = k$, and since the labels on Coxeter diagrams are between 2 and 5, the only option is that $\ell = 2, k = 4$.
Additionally, any connected Coxeter diagram can have at most one label equal to 4, and if it has such a label then every other label is a 2 or 3.

But the diagram is connected, and so there must be a vertex $t$ adjacent to $v$, with $vt$ of order 3.
So applying $(\star)$ to $v$ and then $t$ gives a face which is either a triangle or hexagon, according as $f(t) = 0$ or $1$.
Regardless, we saw that $P$ has a square or octogonal face, and so it cannot be regular.
\end{proof}

With a strong restriction in place in the case when $D$ is connected, it is natural to consider what occurs in the case where $D$ is disconnected.
The previous proof illustrates what might happen: other components may be able to generate square faces, removing the issue with having more than one mirror $v$ with $f(v) = 1$.
Note that Wythoff's construction distributes across disconnected $\{0,1\}$-decorated Coxeter diagrams, in the sense that the Wythoffian polytope of the union is the Cartesian product of the Wythoffian polytopes of the connected components of the diagram.
Also, to be non-degenerate it must be the case that each connected component does not have decoration identically zero.

\begin{lemma}
\label{lem:disconnected}
The only regular polytope represented by a disconnected decorated Coxeter diagram $f : V(D) \to \{0,1\}$ is a hypercube.
\end{lemma}
\begin{proof}
Let $P$ be the Wythoffian polytope corresponding to $f$.
Take $X$ to be a connected component of the diagram, say of size $k$, and denote by $P_X$ the Wythoffian polytope corresponding to the restricted decoration $f|_X$.
If $k = n - 1$, then by applying the rule $(\star)$ to a vertex in $X$ and the vertex not in $X$, it follows that $P_X$ has a square face.
By regularity, every face of $P$, and therefore of $P_X$, is square.
Let $v$ be a vertex of $X$ with $f(v) = 1$.
Applying $(\star)$ to $v$ and then any adjacent vertex must produce a square face, so every neighbour $w$ of $v$ has $f(w) = 0$, and moreover $wv$ has order 4.
The only Coxeter diagrams having label 4 are $F_4$ and $B_n$ for various $n$.
But $f|_X$ cannot be a decoration of $F_4$, because there is no vertex incident to only edges with label 4.
Therefore, $f|_X$ must be the following decoration, which corresponds to the $k$-dimensional hypercube:

\begin{figure}[h]
\centering
\begin{tikzpicture}
\draw (1,0) -- (2.5,0);
\node at (3,0) {$\cdots$};
\draw (3.5, 0) -- (5,0);
\Squarecross{1}
\Squarecross{2}
\Squarecross{4}
\Square{5}
\node at (4.5, 0.125) {\tiny 4};
\end{tikzpicture}
\end{figure}

Otherwise, $k + 1 < n$, and by deleting $n-k-1$ vertices not in $X$, we obtain a decorated Coxeter diagram $Y$ that is the disjoint union of $X$, with an additional isolated vertex $u$ satisfying $f(u) = 1$.
Denote by $P_Y$ the polytope represented by $Y$.
Since $P_Y$ is a $(k+1)$-face of the regular polytope $P$, it is regular, and since $Y$ is disconnected and has size $k+1 < n$, we see inductively that $P_Y$ is a hypercube.
Delete $u$ from $Y$ to see that $P_X$ is a facet of $P_Y$.
Since the facets of hypercubes are lower dimensional hypercubes, $P_X$ is also a hypercube.

In this fashion, we find that every connected component of the diagram corresponds to a hypercube of dimension equal to its size.
Then $P$ is the product of the respective hypercubes, which is again a hypercube.
\end{proof}

\begin{theorem}
\label{theorem:regclass}
The regular polytopes are comprised of the infinite families of simplices, hypercubes, hyperoctahedra, and regular polygons, as well as five exceptional structures: the icosahedron, dodecahedron, 120-cell, 600-cell, and 24-cell.
\end{theorem}
\begin{proof}

By Theorem~\ref{theorem:regularwythoff}, it suffices to classify the Wythoff constructions that are regular.
Moreover, by Lemma~\ref{lem:disconnected} we can without loss of generality take our Coxeter diagram to be connected.
Thus, it must be one of the diagrams from Fig.~\ref{fig:coxeterdiagram}.

All non-empty decorations of $I_2(k)$ give a regular polytope for any $k \in \Z^+$; \drawit{2}{k}{1}{0} gives the regular $k$-gon, and \drawit{2}{k}{1,2}{0} gives the regular $2k$-gon.
These are all the 2-dimensional $\{0,1\}$-Coxeter diagrams.
Henceforth, we turn our attention to connected Coxeter diagrams of size at least 3.
By Lemma~\ref{lem:oneuncrossed}, it suffices to consider the decorations $f$ having $|f\inv(1)| = 1$.

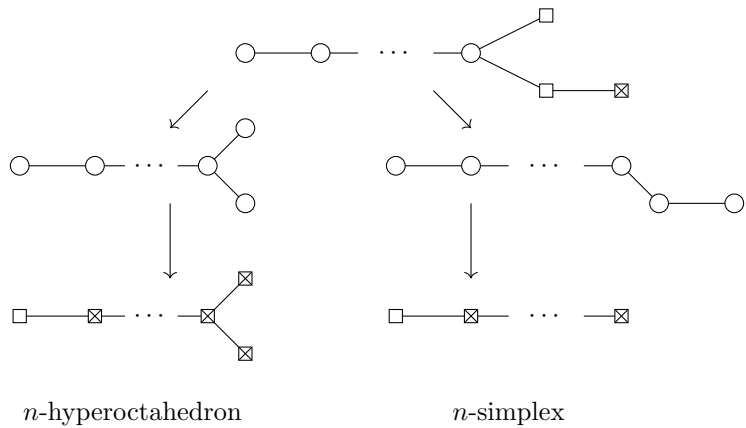
\begin{figure}[h]
\centering
\begin{tikzpicture}
\draw (0,0) -- (1.5,0);
\draw (2.5,0) -- (3,0);
\draw (3,0) -- (4,0.5);
\draw (3,0) -- (4,-0.5);
\draw (4,-0.5) -- (5,-0.5);
\node at (2,0) {$\cdots$};
\SquarE{4}{0.5}
\SquarE{4}{-0.5}
\Squarecrossed{5}{-0.5}
\Circle{0}
\Circle{1}
\Circle{3}

\draw (-3,-1.5) -- (-1.6,-1.5);
\draw (-0.9,-1.5) -- (-0.5,-1.5);
\draw (-0.5, -1.5) -- (0, -1);
\draw (-0.5, -1.5) -- (0, -2);
\node at (-1.25,-1.5) {$\cdots$};
\CirclE{-2}{-1.5}
\CirclE{-0.5}{-1.5}
\CirclE{0}{-1}
\CirclE{0}{-2}
\CirclE{-3}{-1.5}
\draw (-3,-3.5) -- (-1.6,-3.5);
\draw (-0.9,-3.5) -- (-0.5,-3.5);
\draw (-0.5, -3.5) -- (0, -3);
\draw (-0.5, -3.5) -- (0, -4);
\node at (-1.25,-3.5) {$\cdots$};
\SquarE{-3}{-3.5}
\Squarecrossed{-2}{-3.5}
\Squarecrossed{-0.5}{-3.5}
\Squarecrossed{0}{-3}
\Squarecrossed{0}{-4}
\node at (-1.5,-4.8) {$n$-hyperoctahedron};

\draw (2, -1.5) -- (3.5, -1.5);
\draw (4.5, -1.5) -- (5,-1.5) -- (5.5,-2) -- (6.5,-2);
\node at (4,-1.5) {$\cdots$};
\CirclE{2}{-1.5}
\CirclE{3}{-1.5}
\CirclE{5}{-1.5}
\CirclE{5.5}{-2}
\CirclE{6.5}{-2}

\draw (2, -3.5) -- (3.5, -3.5);
\draw (4.5, -3.5) -- (5, -3.5);
\SquarE{2}{-3.5}
\Squarecrossed{3}{-3.5}
\node at (4,-3.5) {$\cdots$};
\Squarecrossed{5}{-3.5}
\node at (3.5, -4.8) {$n$-simplex};

\draw [->] (-0.5, -0.5) -- (-1, -1);
\draw [->] (2.5, -0.5) -- (3, -1);
\draw [->] (-1, -2) -- (-1, -3);
\draw [->] (3, -2) -- (3, -3);
\end{tikzpicture}
\caption{Two distinct faces of a putative regular polytope of $E_k$ symmetry for some $k$}
\label{fig:EnBranches}
\end{figure}

Suppose that the diagram has a branch point, as in the case of $D_n (n \ge 4), E_6, E_7, E_8$.
Take the generator $v$ with $f(v) = 1$ in the diagram, and starting with $v$, repeatedly apply the rule $(\star)$ to vertices going in the direction of the branch point.
When the branch point is reached, there are two cases.
If one branch has length longer than two, then continuing to apply $(\star)$ to the first two vertices of each branch and applying $(\star)$ to the first two vertices of the longer branch gives two different decorations; one corresponding to a simplex, and the other a hyperoctahedron.
In particular, no decoration of $E_6, E_7, E_8$ can be regular---as each branch point has two branches of length at least two.
For an illustration of this process, see Fig.~\ref{fig:EnBranches}.

In the case of $D_n, n \ge 5$, since both small branches of $D_n$ have length 1, $v$ must always lie on the long branch.
In fact, $v$ must be on the end of the long branch; otherwise, once the branch point is reached, applying $(\star)$ one of the two branches and then a vertex behind the initial one gives a different decoration than applying $(\star)$ to the vertices on each branch; see, for example, Fig.~\ref{fig:D6branches}.
One may check that the resulting decoration of $D_n$ gives the $n$-dimensional hyperoctahedron.
In the case where $n = 4$, one additional decoration giving a regular polytope is possible, where $v$ is the center vertex.
This diagram corresponds to the 24-cell.

\begin{figure}[h]
\centering
\begin{tikzpicture}
\draw (0,0) -- (1,0) -- (2,0) -- (3,0) -- (4,0) -- (3,0) -- (3,1);
\node[draw,rectangle,fill=white,inner sep=1.5pt] at (0,0) {$b_3$};
\node[draw,circle,fill=white,inner sep=1.5pt] at (1,0) {$a$};
\node[draw,circle,fill=white] at (2,0) {};
\node[draw,circle,fill=white] at (3,0) {};
\node[draw,rectangle,fill=white,inner sep=1.5pt] at (4,0) {$b_1$};
\node[draw,rectangle,fill=white,inner sep=1.5pt] at (3,1) {$b_2$};
\end{tikzpicture}

\caption{In an decoration of $D_6$ having only the $a$ vertex uncrossed, one sequence of transformations leads us to this step.
In the next two steps, applying $(\star)$ to $b_2$ and $b_1$ gives a different decoration than applying $(\star)$ to $b_1$ and $b_3$}
\label{fig:D6branches}
\end{figure}
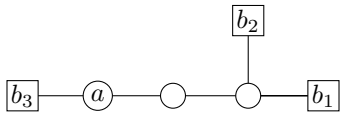

When the diagram does not have a branch point, it must be a path.
In the case of $A_n$, when $n > 3$ the vertex $v$ with $f(v) = 1$ must be on one of the ends.
Otherwise we can form two distinct cells, corresponding to \drawit{3}{}{1}{0} (a tetrahedron) and \drawit{3}{}{2}{0} (an octahedron), by applying $(\star)$ to the vertices on either side of $v$, or by applying it to two on one side of $v$.
The decoration of $A_n$ with $f(v) = 1$ where $v$ is a terminal node of the path, and $f(w) = 0$ for each other $w$ corresponds to the $n$-simplex.
The only remaining case is when $n=3$ with the decoration \drawit{3}{}{2}{0}; we already remarked that this polytope is an octahedron.

If the diagram is not $A_n$, then numbering the vertices in this path $v_1, v_2, \ldots, v_n$, we have some $j$ so that $m_{j-1,j} > 3$, where $m_{j-1,j}$ is the label on the edge (the order of the group element $v_{j-1} v_j$).
Let $i$ be such that $f(v_i) = 1$, and suppose that $v_i$ is not at one of the ends of the path.
Up to possibly relabeling the vertices of the path, we may assume that $1 < i < j$.
Since each connected Coxeter diagram has at most one label not a 2 or 3, we obtain two diagrams: one by applying $(\star)$ to vertices $i, \ldots, j$, and another by applying $(\star)$ to vertices $i-1, \ldots, j-1$.
If the polytope is to be regular, these $(j-i+1)$-faces must be the same and both themselves regular.
The case of $A_n$ handled above shows that the only way this condition can occur is if $j-i+1 = 3$, and the cells \drawit{3}{$ $}{2}{0} and \drawit{3}{$ $,m_{j-1,j}}{1}{0} are the same.
Since \drawit{3}{$ $}{2}{0} is an octahedron, we must have $m_{j-1,j} = 4$.
The only possible decoration satisfying this information is \drawit{4}{$ $,$ $,4}{2}{0}, which corresponds to the 24-cell.

Otherwise, the only decorations of $H_3, H_4, F_4, B_n$ representing regular polytopes must have the vertex $v$ with $f(v) = 1$ occuring at one of the ends of the path.
All such decorations do give regular polytopes: \drawit{3}{5}{3}{0} is the icosahedron, \drawit{3}{5}{1}{0} is the dodecahedron, \drawit{4}{5}{4}{0} is the 600-cell, \drawit{4}{5}{1}{0} is the 120-cell, and \drawit{4}{$ $,4}{1}{0} is the 24-cell.
For a decoration of $B_n$, when $v$ is on the edge with a 4, the result is the $n$-hypercube.
If it is instead on the edge marked with a 3, it is the $n$-hyperoctahedron.
\end{proof}

\bibliographystyle{ieeetr}
\bibliography{new_draft}

\begin{thebibliography}{1}

\bibitem{Champagne1995}
B.~Champagne, M.~Kjiri, J.~Patera, and R.~T. Sharp, ``Description of
  reflection-generated polytopes using decorated {Coxeter} diagrams,'' {\em
  Canadian J. Physics}, vol.~73, pp.~566--584, 1995.

\bibitem{Coxeter-polytopes}
H.~S.~M. Coxeter, {\em Regular polytopes}.
\newblock New York: Dover Publications Inc., third~ed., 1973.

\bibitem{Ziegler1995}
G.~M. Ziegler, {\em Lectures on polytopes}, vol.~152 of {\em Graduate Texts in
  Mathematics}.
\newblock Springer-Verlag New York, 1995.

\bibitem{Abstract-Regular-Polytopes}
P.~McMullen and E.~Schulte, {\em Abstract Regular Polytopes}.
\newblock Cambridge University Press, 2002.

\bibitem{Coxeter-WythoffConstruction}
H.~S.~M. Coxeter, ``Wythoff's {C}onstruction for {U}niform {P}olytopes,'' {\em
  Proc. London Math. Soc. (2)}, vol.~38, pp.~327--339, 1935.

\bibitem{Conway-Guy-Four-Dimensional-Archimedean}
J.~H. Conway and M.~J.~T. Guy, ``Four-dimensional {Archimedean} polytopes,'' in
  {\em Proceedings of the Colloquium on Convexity, Copenhagen, 1965},
  (Copenhagen, Denmark), K{\o}benhavns Universitets Matematiske Institut, 1965.

\bibitem{Coxeter-DiscreteGroupsGeneratedByReflections}
H.~S.~M. Coxeter, ``Discrete groups generated by reflections,'' {\em Annals of
  Mathematics}, vol.~35, no.~3, pp.~588--621, 1934.

\end{thebibliography}

\end{document}